\documentclass[12pt]{amsart}

\usepackage{amsthm}
\usepackage{amsmath,amssymb}
\usepackage{geometry}

\begin{document}

\newtheorem{theorem}{Theorem}
\newtheorem{proposition}[theorem]{Proposition}
\newtheorem{corollary}[theorem]{Corollary}
\newtheorem{lemma}[theorem]{Lemma}
\newtheorem{definition}[theorem]{Definition}
\newtheorem{remark}[theorem]{Remark}
\newtheorem{conjecture}[theorem]{Conjecture}

\newcommand\Z{{\mathbb Z}}
\newcommand\K{{\mathbb K}}
\newcommand\R{{\mathbb R}}
\newcommand\Sp{{\mathbb S}}
\newcommand\A{{\mathcal A}}
\newcommand\AV{{\mathcal {AV}}}
\newcommand\F{{\mathcal F}}
\newcommand\V{{\mathcal V}}
\newcommand\LL{{\mathcal L}}
\newcommand\dd[1]{\frac{\partial}{\partial #1}}
\newcommand\Ah{\A_{(h)}}
\newcommand\Der{\textup{Der\,}}
\newcommand\Stab{\textup{Stab}}
\newcommand\Span{\textup{Span}}
\newcommand\gl{\mathfrak{gl}}
\newcommand\sll{\mathfrak{sl}}
\newcommand\NN{{\mathcal N}}
\newcommand\bi{{\overline i}}
\newcommand\ti{{\widetilde i}}
\newcommand\tl{{\widetilde l}}

\title
[Gauge modules]
{Gauge modules for the Lie algebras of vector fields on affine varieties}
\author{Yuly Billig}
\address{School of Mathematics and Statistics, Carleton University, Ottawa, Canada}
\email{billig@math.carleton.ca}
\author{Jonathan Nilsson}
\address{Mathematical Sciences, Chalmers University of Technology, Sweden}
\email{jonathn@chalmers.se}
\author{Andr\'e Zaidan}
\address{Instituto de Matem\'atica e Estat\'istica, Universidade de S\~ao Paulo, Brazil}
\email{andrezaidan@gmail.com}

\maketitle

\begin{abstract}
For a smooth irreducible affine algebraic variety we study
a class of gauge modules admitting compatible actions
of both the algebra $A$ of functions and the Lie algebra
$\mathcal{V}$ of vector fields on the variety.
We prove that a gauge module
corresponding to a simple $\mathfrak{gl}_N$-module is irreducible
as a module over the Lie algebra of vector fields unless
it appears in the de Rham complex.
\end{abstract}

\section{Introduction}
David Jordan proved simplicity of an important class of infinite-dimensional Lie algebras -- Lie algebras of vector fields on smooth irreducible affine varieties 
\cite{J1} (see also \cite{J2} and \cite{BF3}). The structure of these algebras is very different from the case of simple finite-dimensional Lie algebras, it was shown in 
\cite{BF3} that Lie algebras of vector fields may contain no non-zero nilpotent or semisimple elements. For this reason, standard tools of roots and weights are not applicable in 
representation theory of this class of Lie algebras. The development of representation theory for the Lie algebras of vector fields on affine algebraic varieties was initiated in 
\cite{BN} and \cite{BFN}. These papers proposed to study a category of finite rank $\AV$-modules, i.e. modules that in addition to the action of the 
Lie algebra $\V$ of vector fields on an affine variety $X$ of dimension $N$, admit a compatible action of the commutative algebra $\A$ of polynomial functions on $X$ and 
are finitely generated as $\A$-modules.

 Motivated by a non-abelian gauge theory, \cite{BFN} introduced a family of finite rank $\AV$-modules, called gauge modules. There are two ingredients in the construction 
of a gauge module -- a finite-dimensional representation $U$ of a Lie algebra $\LL_+$, which is a subalgebra spanned by elements of non-negative degrees in 
$\textup{Der}(\K[t_1,\ldots,t_N])$, and the gauge fields $\{ B_i \}$ (see Section 3 for details). The main result of \cite{BFN} states that if $U$ is irreducible 
(in which case it is just a simple finite-dimensional irreducible $\mathfrak{gl}_N$-module), then the corresponding gauge module is irreducible as an $\AV$-module.

 The goal of the present paper is to investigate irreducibility of simple gauge $\AV$-modules as modules over the Lie algebra $\V$ of vector fields. We prove that a gauge 
 $\AV$-module corresponding to an irreducible finite-dimensional $\mathfrak{gl}_N$-module $U$ remains irreducible as a $\V$-module, unless $U$ is an exterior power of a natural 
$N$-dimensional $\mathfrak{gl}_N$-module. Exceptional $\AV$-modules appear in the de Rham complex, 
whereas the de Rham differential is a homomorphism of $\V$-modules (but not $\AV$-modules). 
For this reason, kernels and images of the de Rham differential are $\V$-submodules in the corresponding gauge modules. This result is a generalization of a theorem of Eswara Rao
\cite{E1} on irreducible tensor modules over the Lie algebra of vector fields on a torus.

 The idea of our proof is to recover $\A$-action from $\V$-action, trying to show that every $\V$-submodule in a gauge module is an $\AV$-submodule. This strategy fails precisely 
in the case of de Rham modules. One of the main technical tools we use is Hilbert's Nullstellensatz.

 We also show that simplicity of $U$ as a $\gl_N$-module is not a necessary condition for irreducibility of a gauge $\V$-module. We construct an example of a simple gauge module over the Lie algebra $W_1$ of vector fields on a circle with a $\gl_1$-module $U$ which decomposes in a direct sum of two submodules.
To prove irreducibility of this gauge module for $W_1$, we analyze the structure of this module as an
$\sll_2$-module.

 The structure of the paper is as follows. In Section 2 we recall the basics of the Lie algebras of vector fields on affine varieties and define the category of $\AV$-modules. 
In Section 3 we discuss the construction of gauge modules and prove that there could be at most $N+1$ exceptional simple $\mathfrak{gl}_N$-modules $U$ for which irreducibility over $\V$ 
fails. In Section 4 we discuss de Rham complex of gauge modules and show that exceptional $\mathfrak{gl}_N$-modules are exterior powers of the natural $N$-dimensional module. In Section 5 we explore connections between $\sll_2$-modules and gauge modules for $W_1$.

\subsection*{Acknowledgements}
Research of Y.B. is supported with a grant from the
Natural Sciences and Engineering Research Council of Canada. A.~Z. was partially supported by CAPES process 88881.133701/2016-01. J.N. and
A.Z. thank Carleton University for the hospitality during their visits.

\section{Definitions and Notations}

Let $\K$ be an algebraically closed field of characteristic 0, and let $X\subset \mathbb{A}_\K^n$ be a smooth
irreducible affine variety of dimension $N$. Let $I = \langle g_1,\ldots,g_m\rangle$ be the ideal of functions in $\K[x_1, \ldots, x_n]$ that vanish on $X$, 
let $\A = \K[x_1,\ldots, x_n]/I$ be the algebra of polynomial functions on $X$, and let $\mathcal{V} = \textup{Der}(\A)$ be the Lie algebra of vector fields on $X$.

An $\A\mathcal{V}$-module is a vector space $M$ equipped with module structures over both the commutative unital algebra $\A$ and over the Lie algebra $\mathcal{V}$ such that these structures are compatible in the following sense:

$$\eta\cdot(f\cdot m)=\eta(f)\cdot m +f\cdot (\eta\cdot m),$$
for all $\eta \in \mathcal{V}$, $f \in \A$, and $m \in M$.

We say that an $\AV$-module $M$ has finite rank if it is finitely generated as an $\A$-module.

Let $J$ be the Jacobian matrix $J = (\frac{\partial g_j}{\partial x_i})$. Let $r$ be the rank of $J$ over the field $\F$ of rational functions on $X$. Then $N = $ dim $X = n - r$.
Let $\{h_i\}$ be the set of non-zero $r\times r$ minors of $J$ and let $N(h_i) = \{ p \in X \,|\, h_i(p) \neq 0\}$. 
The Jacobian criterion for smoothness (see e.g.~\cite[Section I.5]{Ha}) states that
$\bigcup_i N(h_i) = X$. 
The Lie algebra $\V$ can be described as an $\A$-submodule of a free $\A$-module $\oplus_{i=1}^n \A \frac{\partial}{\partial x_i}$,
which is the kernel of the Jacobian matrix:
$\sum_{i=1}^n f_i\frac{\partial}{\partial x_i} \in \mathcal{V}$ 
if and only if $\sum_{i=1}^n f_i\frac{\partial g_j}{\partial x_i} = 0$ in $\A$ for all $j = 1, \ldots, m$ (see e.g. \cite{BF3}).

Fix a non-zero minor $h$ of $J$, let $\beta \subset \{1,\ldots,n\}$ be the set of columns of $J$ in $h$. Since rank $J$ is equal to $|\beta|$, 
we can solve the this system of linear equations over $\F$ treating $f_i, i \not\in \beta$ as free variables, and construct solutions for each $i \not\in \beta$
$$\tau_i=\frac{\partial}{\partial x_i}+\sum_{j \in \beta}f_{ij}\frac{\partial}{\partial x_j},$$
where $f_{ij} \in h^{-1} \A$, and hence $h\tau_i \in \mathrm{Der}(\A)$. Note that each $\tau_i$ is a derivation of the localized algebra $\A_{(h)}$ but not necessarily of $\A$.

We have the following definition from \cite{BN}:

\begin{definition}
We shall say that $t_1,\ldots, t_N \in \A$ are chart parameters in
the chart $N(h)$ provided that the following conditions are satisfied:
\begin{enumerate}
\item $t_1,\ldots, t_N$ are algebraically independent, so $\K[t_1,\ldots, t_N] \subset \A$.
\item Each element of $\A$ is algebraic over $\K[t_1,\ldots, t_N]$.
\item For each $i = 1, \ldots, N$, the derivation $\frac{\partial}{\partial t_i} \in \textup{Der}(\K[t_1,\ldots, t_N])$ extends to a derivation of the localized algebra $\A_{(h)}$.
\end{enumerate}
\end{definition}

This definition implies that 
$$\textup{Der}(\A_{(h)}) = \mathop\oplus_{i=1}^N \A_{(h)}\frac{\partial}{\partial t_i}.$$ 
Since $\textup{Der}(\A) \subset \textup{Der}(\A_{(h)})$,
each polynomial vector field $\eta$ on $X$ can be written as 
\[\eta=\sum_{i=1}^N f_i \frac{\partial}{\partial t_i}\]
where $f_i = \eta(t_i) \in \A$. Note that here $\frac{\partial}{\partial t_i}$ is interpreted as the unique extension of the partial derivative on $\K[t_1,\ldots, t_N]$
to a derivation of $\A_{(h)}$.

\begin{lemma}[\cite{BN}, Lemma 4]  We have that $\{x_i \, | \, i \notin \beta\}$ are chart parameters in the chart $N(h)$.
\end{lemma}

In this case a derivation
$$\eta = \sum_{i\notin\beta}f_i\frac{\partial }{\partial x_i}+\sum_{j\in\beta}f_j\frac{\partial }{\partial x_j}$$
when embedded in $\textup{Der}(\A_{(h)})$ will be written simply as 
$$\eta = \sum_{i\notin\beta} f_i \frac{\partial}{\partial x_i}$$ 
with the understanding that for $j \in \beta$ we have $\eta(x_j) = \sum_{i\notin\beta}f_i \partial x_j/\partial x_i =\sum_{i\notin\beta}f_i\tau_i(x_j)$. 
With this convention we have $h\frac{\partial}{\partial x_i} \in \mathcal{V}$ for every $i \notin\beta$.

\hspace{0cm}

\textbf{Example.} Let us take $X$ to be a 2-dimensional sphere $\mathbb{S}^2\subset \mathbb{A}^3$ with a defining ideal $I =\langle x^2+y^2+z^2-1\rangle$. Then the Jacobian matrix is 
$$J = \begin{pmatrix}
2x & 2y & 2z
\end{pmatrix}.$$

We have three minors $h_1 = x, h_2 = y$ and $h_3 = z$. Each corresponding chart is the sphere without a great circle $x = 0$, $y = 0$ or $z = 0$, respectively. 
Each point of $X$ belongs to at least one chart. Let us fix $h = z$ so that $x$ and $y$ are chart parameters. 
The partial derivative $\frac{\partial}{\partial x}$ of $\K[x,y]$ extends to a derivation of $\A_{(z)}$ as
$\tau_x = \frac{\partial}{\partial x}-\frac{x}{z} \frac{\partial}{\partial z}$, and if we multiply $\tau_x$ by $h = z$ 
we obtain $z\frac{\partial}{\partial x}-x \frac{\partial}{\partial z}$ which is a vector field on $X$. 
Treating $z$ as an implicit function of $x$ and $y$, we will write $\tau_x = \frac{\partial}{\partial x}$ with 
understanding that $\frac{\partial z}{\partial x} = -\frac{x}{z}$.

\section{Gauge Modules}

Let us recall a family of gauge $\AV$-modules that was introduced in \cite{BFN}. Let $\LL = \textup{Der}(\K[t_1,\ldots,t_N])$. This Lie algebra has a natural $\Z$-grading:
$\LL = \LL_{-1} \oplus \LL_0 \oplus \LL_1 \oplus \LL_2 \oplus \ldots$. We set $\LL_+$ to be a subalgebra of $\LL$ of elements of non-negative degree:
$\LL_+ = \LL_0 \oplus \LL_1 \oplus \LL_2 \oplus \ldots$. Note that $\LL_0$ is spanned by the elements $t_i \frac{\partial}{\partial t_j}$ and is isomorphic to $\mathfrak{gl}_N$.

\begin{definition} Let $(U, \rho)$ be a finite-dimensional $\LL_+$-module. \\
 Functions $B_i: \A_{(h)} \otimes U \rightarrow \A_{(h)} \otimes U$, $i = 1,\ldots,N$, are called gauge fields if
\begin{enumerate}
\item each $B_i$ is $\A_{(h)}$-linear,
\item $[B_i, \rho(\LL_+)] = 0$,
\item $[\frac{\partial}{\partial t_i} + B_i,\frac{\partial}{\partial t_j} + B_j ] = 0$ as operators on $\A_{(h)} \otimes U$ for all $i, j\in \{1,\ldots,N\}$.
\end{enumerate}
\end{definition}

Let $\{B_i\}$ be gauge fields, $i = 1,\ldots,N$. Then the space $\A_{(h)} \otimes U$ is an $\textup{Der}(\A_{(h)})$-module with the following action 
$$\left(f\frac{\partial}{\partial t_i}\right)\cdot(g\otimes u) = 
f\frac{\partial g}{\partial t_i}\otimes u +fg\otimes B_iu + 
\sum_{k \in \Z_+^N\backslash \{ 0 \} } \frac{1}{k!} g\frac{\partial^k f}{\partial t^k} \otimes \rho\left( t^k \dd{t_i} \right) u.$$
where $f, g \in \A_{(h)}$ and $u \in U$.

Identifying the Lie algebra vector fields $\mathcal{V}$ with its natural embedding into $\textup{Der}(\A_{(h)})$ and with the left $\A$-action by multiplication, we have that $\A_{(h)} \otimes U$ has the structure of an $\A\mathcal{V}$-module.

\begin{definition} An $\A\mathcal{V}$-submodule of $\A_{(h)}\otimes U$ which has finite rank over $\A$ 
is called a local gauge module. We say that an $\AV$-module M is a gauge module if it is isomorphic to a local gauge module for each chart $N(h)$ in our standard atlas.
\end{definition}

In this paper we will focus on the case when $U$ is an irreducible finite-dimensional $\LL_+$-module. 
It was shown in \cite{Bi} that such modules are just irreducible $\LL_0 \cong \gl_n$-modules on which $\LL_s$ with $s \geq 1$ act trivially. 
In this case the second axiom implies that the gauge fields $B_i$ are just functions in $\Ah$, and the third axiom of the gauge fields becomes 
$\frac{ \partial B_i} {\partial t_j} = \frac{ \partial B_j} {\partial t_i}$.

The action of $\Der (\Ah)$ on $\Ah \otimes U$ can be written as follows:
\begin{equation}\label{Epi}
\left(f\frac{\partial}{\partial t_i}\right)\cdot(g\otimes u) = 
f\frac{\partial g}{\partial t_i}\otimes u +fg B_i \otimes u + 
\sum_{p=1}^N  g \frac{\partial f}{\partial t_p} \otimes \rho\left( E_{pi} \right) u.
\end{equation}

Given a gauge module $(M, \varphi)$ and a closed 1-form $\omega \in \Omega^1 (X)$ on $X$,
we can define a new gauge module structure $\varphi_\omega$ on the space $M$.
Write $\omega$ in each chart $N(h)$ as
$\omega = P_1 dt_1 + \ldots + P_N dt_N$ with $P_i \in \Ah$. Since $\omega$ is closed, the functions $\{ P_i\}$ will satisfy
$\frac{ \partial P_i} {\partial t_j} = \frac{ \partial P_j} {\partial t_i}$. Then we define 
$\varphi_\omega$ in the chart $N(h)$ as
$$\varphi_\omega \left(f\frac{\partial}{\partial t_i}\right)\cdot m 
= \varphi \left(f\frac{\partial}{\partial t_i}\right)\cdot m + f P_i m .$$

If the form $\omega$ is exact, $\omega = d(G)$ for some $G \in \A$, then module $(M , \varphi_\omega)$ may be formally interpreted as the space $e^G M$ with the ``old'' action $\varphi$:
$$\varphi \left(f\frac{\partial}{\partial t_i}\right)\cdot e^G m
= e^G \varphi \left(f\frac{\partial}{\partial t_i}\right)\cdot m + e^G f \frac{ \partial G} {\partial t_j} m. $$

Of course, if $H^1_{dR} (X) \neq 0$ then not every closed 1-form is exact, hence we can not always interpret $\varphi_\omega$ as a formal shift of the algebra of functions by a formal factor $e^G$. 

Although we can construct new gauge modules by modifying the action with a help of a closed 1-form, an example of a family of rank 1 gauge modules for $X = \Sp^2$, given in \cite{BFN}, shows that 
we can not obtain all gauge modules in this way, starting from modules with zero gauge fields. 

We recall the main theorem of~\cite{BFN}:
\begin{theorem}\label{thm:bifuni}\cite[Theorem 24]{BFN} Let $X$ be a smooth irreducible affine algebraic variety and let $M$ be a gauge module which corresponds 
to a simple finite-dimensional $\mathfrak{gl}_N$-module $U$. Then $M$ is a simple $\A\mathcal{V}$-module.
\end{theorem}

The main goal of this paper is to investigate simplicity of these modules over $\V$. We are going to show that gauge modules remain irreducible as $\V$-modules, 
unless $U$ is an exterior power of the natural $N$-dimensional $\gl_N$-module.
This will be done by reconstructing the $\A$-action from the $\mathcal{V}$-action.

For the rest of the paper we will assume that $M$ is a gauge module which corresponds 
to a simple finite-dimensional $\mathfrak{gl}_N$-module $U$.

We shall fix the standard generators of the centre of $U(\mathfrak{gl}_N)$. Letting
\[\Omega_k = \sum_{i_1,\ldots, i_k=1}^N E_{i_1i_2}E_{i_2i_3}\cdots E_{i_ki_1}\] we have $Z(\mathfrak{gl}_N)=\K[\Omega_1,\ldots, \Omega_N]$. 
\begin{lemma}\label{lem:omega1}
 If a gauge module $M$ is reducible as a $\mathcal{V}$-module, then $\Omega_1$ acts on $U$ by a scalar from the set $\{0,1,\ldots,N\}$.
\end{lemma}
\begin{proof}
Let $M^\prime$ be a nontrivial $\V$-submodule in $M$. 
Let us fix one of the charts $\NN(h)$ of $X$ with its chart parameters.
Let $f \in\A$, $i \in \{1,2,\ldots,N\}$, $c \in k$, and $0 \leq r \leq 2$.
Consider the composition of the actions of vector fields from $\V$ on $\Ah \otimes U$:
\[((t_i+c)^{2-r} f h\frac{\partial}{\partial t_i}) \circ ((t_i+c)^r h\frac{\partial}{\partial t_i}) (g \otimes u). \]
When we expand this expression using (\ref{Epi}), we will get a quadratic polynomial in $r$. Using a Vandermonde determinant argument, we conclude
that $M^\prime$ is invariant under the terms that correspond to each power of $r$.

The operator that corresponds to $-r^2$ is 
\begin{equation*}
(t_i + c)^2 f h^2 \left(E_{ii}^2-E_{ii}\right). 
\end{equation*} 

Since $c$ here is arbitrary, we conclude that $M^\prime$ is invariant under
\begin{equation}
f h^2 \left(E_{ii}^2-E_{ii}\right). \label{Eii} 
\end{equation} 

Since $X$ is smooth, the set of functions $\{ h_j \}$ determining the charts of $X$, has no common zeros on $X$. By Hilbert's Nullstellensatz, 
the ideal generated by $\{ h_j^2 \}$ contains $1$. Hence, $M^\prime$ is invariant under $E_{ii}^2-E_{ii}$, and more generally $f (E_{ii}^2-E_{ii})$ for any $f \in \A$. 

Consider the decomposition of $U$ into the joint eigenspaces for the family of commuting diagonalizable operators $E_{ii}^2-E_{ii}$, $i = 1, \ldots N$. 

For $m \in M^\prime$ consider its expansion in the joint eigenvectors for $E_{ii}^2-E_{ii}$, $m = \sum_\alpha m_\alpha$. 
By a Vandermonde argument, each component $m_\alpha$ is in $M^\prime$. 

Let us assume that the zero eigenspace is trivial, it means that
for each component $m_\alpha$ there exists at least one $i$ such that $E_{ii}^2-E_{ii}$ is acting by a non-zero scalar, 
hence for every function $f \in \A$ we get $f m_\alpha \in M^\prime$. This implies that $f m \in M^\prime$ and $M^\prime$ is in fact an $\AV$-submodule in $M$.
But Theorem \ref{thm:bifuni} states that $M$ is irreducible as an $\AV$-module. We conclude that the zero eigenspace is in fact nontrivial. 
This means that each $E_{ii}$ is acting on this space as either $0$ or $1$. 
But then $\Omega_1 = \sum_{i=1}^N E_{ii}$ acts on $U$ by a scalar from $\{0,1,\ldots,N\}$.
\end{proof}

We will use the notation $I_N = \{1, 2, \ldots, N\}$ below.

\begin{proposition} \label{submoduleinvariant}
Let $k\geq 2$. For every $f \in \A$ any $\V$-submodule $M^\prime$ in $M$ is invariant under the action of
$$f  \left( \sum_{i \in I_N^k} \sum_{\sigma \in S_k}E_{i_{\sigma(1)} i_1}E_{i_{\sigma(2)} i_2}\ldots E_{i_{\sigma(k)} i_k}-\frac{\left(N+k-1\right)!}{N!}\Omega_1\right).$$
\end{proposition}
\begin{proof}
 Fix $s$ to be a large enough element in $\Z_+^N$ (it should be large enough for the Vandermonde arguments we use below to work, $s_j \geq k^3$ for all $j$ will suffice).
Let $c \in \K^N$ and $f \in \A$. Let $r \in \Z_+^N$ with $0 \leq r_j \leq k$ for all $1\leq j \leq N$.
Denote by $\{ \epsilon_1, \ldots, \epsilon_N\}$ is the standard basis of $\Z^N$.

Consider the action of
\[((t+c)^{s + \sum_{l=1}^k\epsilon_{i_l} - (n_2 + \ldots + n_k) r} f h\frac{\partial}{\partial x_{i_1}}) \circ ((t+c)^{n_2 r}h\frac{\partial}{\partial x_{i_2}}) \circ \cdots 
\circ ((t+c)^{n_k r}h\frac{\partial}{\partial x_{i_k}}).\]

As in the Lemma \ref{lem:omega1} we can expand this expression into a polynomial in $n_2, \ldots, n_k$ by applying (\ref{Epi}).
By a Vandermonde argument, $M^\prime$ is invariant under the action of each term in this polynomial. 
The action of the term corresponding to the monomial $n_2^2 n_3\ldots n_k$ is given by

\begin{equation*}
(t+c)^s f h^k \left(r_{i_1}-\sum_{q_1} r_{q_1}E_{q_1 i_1}\right)\left(\sum_{q_2} r_{q_2}E_{q_2 i_2}\right)\ldots\left(\sum_{q_k} r_{q_k}E_{q_k i_k}\right)
\end{equation*}

Since $c$ is arbitrary, $M^\prime$ is also invariant under
\begin{equation}
f h^k \left(r_{i_1}-\sum_{q_1} r_{q_1}E_{q_1 i_1}\right)\left(\sum_{q_2} r_{q_2}E_{q_2 i_2}\right)\ldots\left(\sum_{q_k} r_{q_k}E_{q_k i_k}\right)\label{js}
\end{equation}

Consider a sequence $i = (i_1, \ldots, i_k) \in I_N^k$. The permutation group $S_k$ acts naturally on the set $I_N^k$. Let us denote by $\Stab(i)$ the stabilizer of $i$ in $S_k$.
Set 
$$\omega_i = \frac{1}{|\Stab(i)|}\sum_{\sigma \in S_k} E_{i_{\sigma(1)} i_1}E_{i_{\sigma(2)} i_2} \ldots E_{i_{\sigma(k)} i_k}.$$ 
For a sequence $i \in I_N^k$ denote by $\bi$ the truncated sequence $\bi = (i_2, \ldots, i_k) \in I_N^{k-1}$. 
Then the coefficient of $-r_{i_1} r_{i_2} \ldots r_{i_k}$ in (\ref{js}) is $f h^k (\omega_i - \omega_\bi)$ and $M^\prime$ is invariant under the action of this operator.
By recursion on $k$ we conclude that $M^\prime$ is invariant under $f h^k (\omega_i - E_{i_k i_k})$.

Multiplying these expressions by $|\Stab(i)|$ and taking a sum over all $i$, we get that $M^\prime$ is invariant under 
$$f h^k \left( \sum_{i \in I_N^k} \sum_{\sigma \in S_k} E_{i_{\sigma(1)} i_1} E_{i_{\sigma(2)} i_2} \ldots E_{i_{\sigma(k)} i_k}-\frac{1}{N}\sum_{i \in I_N^k} |Stab(i)| \Omega_1\right).$$

Let $\{O_j\}$ be the set of orbits of $S_k$ in $I_N^k$. Each orbit can be thought as a $k$-combination with repetitions from a set of $N$ elements. 
By standard combinatorics the number of the orbits is $\binom{N+k-1}{k}$. Fix a representative $i_j$ in each orbit $O_j$.

Then we have:
$$\sum_i|\Stab(i)| = \sum_j |O_j|\cdot|\Stab(i_j)| = \sum_j |S_k| = \binom{N+k-1}{k} k!=\frac{(N+k-1)!}{(N-1)!},$$
and $M^\prime$ is invariant under 
$$f h^k \left( \sum_{i \in I_N^k} \sum_{\sigma \in S_k} E_{i_{\sigma(1)} i_1} E_{i_{\sigma(2)} i_2} \ldots E_{i_{\sigma(k)} i_k}- \frac{(N+k-1)!}{N!} \Omega_1\right).$$
Since this is true for each chart of $X$, we can apply Hilbert's Nullstellensatz to the ideal generated by $\{ h_j^k \}$ and drop $h^k$ in the above formula, 
obtaining the claim of the proposition.
\end{proof}

\begin{lemma}\label{belongscenter}
$$\widehat{\Omega}_k = \sum_{i \in I_N^k} \sum_{\sigma \in S_k}E_{i_{\sigma(1)} i_1}E_{i_{\sigma(2)} i_2}\ldots E_{i_{\sigma(k)} i_k} $$
belongs to $Z(U(\mathfrak{gl}_N))$.
\end{lemma}
\begin{proof}
Fix $a, b \in I_N$. We need to show that $\left[ E_{ab} \, , \, \widehat{\Omega}_k \right] = 0$ in $U(\gl_N)$. We can evaluate this commutator as follows:
$$\sum_{\sigma \in S_k} \sum_{l \in I_k} \sum_{i \in I_N^k}  \delta_{b i_{\sigma(l)}} E_{i_{\sigma(1)} i_1} \ldots E_{a i_l} \ldots E_{i_{\sigma(k)} i_k}
- \delta_{a i_l} E_{i_{\sigma(1)} i_1} \ldots E_{i_{\sigma(l)} b} \ldots E_{i_{\sigma(k)} i_k}$$
$$ = \sum_{\sigma \in S_k} \sum_{l \in I_k} \sum_{\genfrac{}{}{0pt}{2}{i \in I_N^k}{i_{\sigma(l)} = b}} E_{i_{\sigma(1)} i_1} \ldots E_{a i_l} \ldots E_{i_{\sigma(k)} i_k}
 - \sum_{\sigma \in S_k} \sum_{l \in I_k} \sum_{\genfrac{}{}{0pt}{2}{i \in I_N^k}{i_l = a}} E_{i_{\sigma(1)} i_1} \ldots E_{i_{\sigma(l)} b} \ldots E_{i_{\sigma(k)} i_k} .$$
Denote the product in the first sum as $\Gamma^+ (\sigma, l, i)$ and the product in the second sum as $\Gamma^- (\sigma, l, i)$.
The terms $\Gamma^- (\sigma, l, i)$ appear only for $i$ with $i_l = a$. For each such $i = (i_1, \ldots, a, \ldots, i_k)$ set $\ti = (i_1, \ldots, b, \ldots, i_k)$ with 
$\ti_l = b$.  Set $\tl = \sigma^{-1} (l)$. Then we claim that for $i$ with $i_l = a$ we have $\Gamma^- (\sigma, l, i) = \Gamma^+ (\sigma, \tl, \ti)$.
Note that $\Gamma^+ (\sigma, \tl, \ti)$ will appear in the summation since $\ti_{\sigma(\tl)} = \ti_l = b$. Indeed both products $\Gamma^+ (\sigma, \tl, \ti)$ and
$\Gamma^- (\sigma, l, i)$ will have $E_{i_{\sigma(l)} b}$ as $l$-th factor and $E_{a i_{\tl}}$ as $\tl$-th factor, and all other factors are the same as well.

Since we have a bijective correspondence between the terms $\Gamma^+$ and $\Gamma^-$, all terms will cancel, showing that the Lie bracket is indeed zero.
\end{proof}

\begin{corollary}
There exist at most $N+1$ simple finite-dimensional $\mathfrak{gl}_N$-modules $U$, for which simple gauge $\AV$-modules become reducible when viewed as modules over $\V$. 
\end{corollary}
\begin{proof}
Lemma \ref{belongscenter} guarantees that $\widehat{\Omega}_k$ can be written as a polynomial in Casimirs $\Omega_i$, we set
$$P_k(\Omega_1,\ldots,\Omega_k) = \widehat{\Omega}_k -\frac{\left(N+k-1\right)!}{N!}\Omega_1.$$

Note that for $2 \leq k \leq N$, the Casimir $\Omega_k$ will occur in $P_k(\Omega_1,\ldots,\Omega_k)$ with a non-zero coefficient, so $P_k$ may be written as
$P_k(\Omega_1,\ldots,\Omega_k) = c \, \Omega_k + Q_k (\Omega_1,\ldots,\Omega_{k-1})$.
Indeed, the expression for $\widehat{\Omega}_k$ contains the term $E_{12} E_{23} \ldots E_{k1}$, while such a term can not come from the products of the Casimirs 
of lower orders.

From Lemma \ref{submoduleinvariant} we know that $M^\prime$
is invariant under the action of $f P_k(\Omega_1,\ldots,\Omega_k)$. 
If any of these polynomials acts as non-zero scalar on $U$ then we can reconstruct the $\A$-action and conclude that $M^\prime$ is an $\AV$-submodule in $M$,
which will contradict simplicity of $M$ as an $\AV$-module.

 If all polynomials $P_k(\Omega_1,\ldots,\Omega_k)$ for $2 \leq k \leq N$ act on $U$ as zero, each one will be fixing the value of one Casimir in terms of 
the values of the lowers Casimirs, so if we fix the action of $\Omega_1$ and all the polynomials are acting by zero we will be fixing the central character of $U$. 
By the Harish-Chandra Isomorphism we have at most one finite dimensional simple $\mathfrak{gl}_N$-module for each central character, see e.g.~\cite[Chapter~1]{Hu}.
Thus we have at most $N+1$ simple $\mathfrak{gl}_N$-modules that give rise to reducible $\mathcal{V}$-modules, 
since by Lemma \ref{lem:omega1} we have $N+1$ possible values for the action of $\Omega_1$. 
\end{proof}

\begin{remark}
Note that the polynomials $P_k(\Omega_1,\ldots,\Omega_k)$ are determined by the $\mathfrak{gl}_N$-module $U$ and do not depend on neither the gauge fields $\{ B_i \}$,
nor on the variety $X$.
\end{remark}

\hspace{0cm}

\textbf{Example.}
 $P_2(\Omega_1,\Omega_2)$ is given by $\Omega_2+\Omega_1^2-(N+1)\Omega_1$. 
 For $N = 2$, if we take $\Omega_1 \in \{0,1,2\}$ and obtain $\Omega_2$ by solving $P_2(\Omega_1,\Omega_2)=0$, 
 we obtain the following three central characters:
\begin{center}
\begin{tabular}{ | c | c | c | c | }
\hline
$\Omega_1$ & 0 & 1 & 2 \\ \hline
$\Omega_2$ & 0 & 2 & 2 \\ \hline
\end{tabular}
\end{center}
which are the central characters of the modules of exterior powers of the natural module $\K^2$.

\section{De Rham Complex}
Let $V$ be the natural $\mathfrak{gl}_N$-module. It has a basis $\{e_1, \ldots ,e_N\}$ on which $\mathfrak{gl}_N$ acts by $E_{ij}\cdot e_k=\delta_{jk}e_i$. 
This action extends naturally to $\Lambda^k V$ in the standard way: for $x\in \mathfrak{gl}_n$ we have
\[x\cdot (e_{i_1}\wedge \cdots \wedge e_{i_k}) = \sum_{j=1}^k e_{i_1}\wedge \cdots  \wedge (x\cdot e_{i_j})\wedge \cdots \wedge e_{i_k}.\]
We also extend this definition to the trivial case and define $\Lambda^{0} V=\K$ to be $1$-dimensional with all of $\mathfrak{gl}_N$ acting as zero. 
We note that each $\Lambda^k V$ is a simple $\mathfrak{gl}_N$ module on which the identity matrix acts by the scalar $k$.

Let $B=(B_1, \ldots, B_N)$ with $B_i \in \A_{(h)}$ such that $\frac{\partial}{\partial x_i}(B_j)=\frac{\partial}{\partial x_j}(B_i)$. 
For example we may pick $B=\nabla G$ for a fixed function $G \in \A$.
Then for $1\leq k \leq N$, we have an $\A\mathcal{V}$-module structure on $\A_{(h)}\otimes \Lambda^k V$ where 
a vector field $\sum f_i \frac{\partial}{\partial x_i}$ as embedded in $\mathrm{Der}(\A_{(h)})$ acts via
\[(\sum_{i=1}^N f_i \frac{\partial}{\partial x_i}) \cdot g\otimes v = \sum_{i=1}^N (f_i \frac{\partial g}{\partial x_i}+B_if_ig)\otimes v + 
\sum_{i,p=1}^N \frac{\partial f_i}{\partial x_p} g \otimes E_{pi}v.\] 

Now we define maps 
\[d_k: \A_{(h)}\otimes \Lambda^k V \rightarrow \A_{(h)}\otimes \Lambda^{k+1} V \quad \text{ by } 
\quad d_k(g\otimes v) = \sum_{p=1}^N (\frac{\partial g}{\partial x_p} + B_pg)\otimes e_p\wedge v.\]

\begin{proposition}\label{cx}
The maps $d_k$ are morphisms of Lie algebra modules which satisfy $d_{k+1} \circ d_{k}=0$. In other words,
\[ \A_{(h)} \otimes \Lambda^{0} V \xrightarrow{\:\; d_0 \:\;} \A_{(h)} \otimes \Lambda^{1} V  \xrightarrow{\:\; d_1 \:\;} \cdots \xrightarrow{d_{N-1}} \A_{(h)} \otimes \Lambda^{N} V \] is a chain complex in the category $\mathcal{V}$-Mod. 
\end{proposition}

Note however that the maps $d_k$ are not $\A$-module morphisms. We can now state the main theorem of our paper.
\begin{theorem}
\label{main}
If $U$ is a simple finite-dimensional $\mathfrak{gl}_N$-module that is not an exterior power of the natural module, then any gauge module $M \subset \A_{(h)}\otimes U$ is simple as a $\V$-module.
\end{theorem}
\begin{proof}
By Lemma~\ref{lem:omega1} a gauge module $M\subset \A_{(h)}\otimes U$ is simple in $\V$-Mod if $\Omega_1=\sum_{i=1}^N E_{ii}$ does not act on $U$ as a scalar belonging from the set $\{0, \ldots, N\}$. Moreover, by the discussion after Lemma~\ref{belongscenter}, there exists at most $N+1$ iso-classes of such exceptional $\mathfrak{gl}_N$-modules which may correspond to $\V$-reducible gauge modules.

It therefore only remains to show that there exist $\V$-reducible gauge modules precisely when $U=\Lambda^k V$ for $0\leq k \leq N$. 
For this we look at the easy example of affine space: take $X=\mathbb{A}^N$, and $h=1$. Then the chart $N(h)$ covers $X$, the standard variables are chart parameters, and $\A_{(h)}=\A=\K[x_1,\ldots,x_N]$. Picking all $B_i=0$, we have a gauge-module structures on each $\A\otimes \Lambda^k V$ for each $0 \leq k \leq N$ where a vector field acts by
\[(\sum_{i=1}^N f_i \frac{\partial}{\partial x_i}) \cdot g\otimes v = \sum_{i=1}^N f_i \frac{\partial g}{\partial x_i}\otimes v + 
\sum_{i,p = 1}^N \frac{\partial f_i}{\partial x_p} g \otimes E_{pi}v.\] 
Since the maps $d_k$ from Proposition~\ref{cx} are $\V$-module homomorphisms, the kernel of $d_{k}$ is a submodule of $\A \otimes \Lambda^k V$ for $0\leq k < N$. Since $d_k(1\otimes e_1 \wedge \cdots \wedge e_k)=0$ but $d_k(x_1\otimes e_2\wedge \cdots \wedge e_{k+1})=e_1\wedge \cdots \wedge e_{k+1} \neq 0$, we see that $\mathrm{Ker}\;d_k$ is in fact a proper submodule, so $\A \otimes \Lambda^k V$ is a reducible $\V$-module for $0 \leq k < N$.

Now only the case $k=N$ remains and we need to find an example of a reducible gauge module-structure on $\A\otimes \Lambda^N V$. In the above example, this module is actually simple, so instead we pick $B_i=-2 x_i$. We claim that 
\[\mathrm{Im}\;d_{N-1}=\{\sum_{i=1}^N (\frac{\partial f_i}{\partial x_i}-2x_if_i)\otimes e_1\wedge \cdots \wedge e_N \;|\; f_1, \ldots, f_N \in \K[x_1,\ldots, x_N]\}\]
 is a proper submodule in $\A\otimes \Lambda^N V$. 
Indeed, this submodule is nonzero and 
does not contain the element $1\otimes e_1 \wedge \cdots \wedge e_N$. 
To see this, we interpret de Rham complex with these gauge fields as 
$$ e^G \Omega^0  \xrightarrow{\:\; d_0 \:\;} \ldots
 \xrightarrow{\:\; d_{N-2} \:\;}  e^G \Omega^{N-1}
 \xrightarrow{\:\; d_{N-1} \:\;}  e^G \Omega^{N},$$
where $G = -x_1^2 - \ldots - x_N^2$.
Let us give an analytic proof that $e^{-x_1^2 - \ldots - x_N^2} dx_1 \wedge \ldots
\wedge dx_N \not\in \mathrm{Im}\;d_{N-1}$ under assumption that $\R \subset \K$. We leave an algebraic proof for a general field $\K$ as an exercise to the reader. 
Let $\omega \in e^G \Omega^{N-1}$. Applying Stokes' theorem, we see that
$$\int_{\R^N} d \omega = 0,$$
however 
$$\int_{\R^N} e^{-x_1^2 - \ldots - x_N^2} dx_1 \wedge \ldots
\wedge dx_N = \pi^{N/2}.$$
Hence, $e^{-x_1^2 - \ldots - x_N^2} dx_1 \wedge \ldots
\wedge dx_N \not\in \mathrm{Im}\;d_{N-1}$.
This concludes the proof.
\end{proof}

\section{Gauge modules on $\Sp^1$ and irreducible modules for $\sll_2$}
In Theorem \ref{main} we proved that gauge modules corresponding to non-exceptional irreducible $\gl_N$ modules $U$, are simple $\V$-modules. In the section we are going to show that irreducibility of $U$ is not a necessary condition for simplicity of a gauge $\V$-module.

In this section we fix $X = \Sp^1$ with equation $x^2 + y^2 = 1$. Setting $t = x + \sqrt{-1}y$,
$s = x - \sqrt{-1}y$, we rewrite the equation of the circle as $t s = 1$. The Jacobian matrix is 
$\begin{pmatrix} s & t \end{pmatrix}$, and we see that the chart $s \neq 0$ covers the whole circle, which allows us to work with a single chart with the chart  parameter $t$. Since $s$ is invertible in $\A$, the localized algebra $\A_{(s)}$ coincides with $\A = \K [t, t^{-1}]$.

For $\alpha \in K$, let us consider the following gauge module for the Lie algebra of vector fields on a circle,
$W_1 = \Der \K [t, t^{-1}]$. Take $U_\alpha = \K^2$ with a basis $\{v, u\}$ and the identity matrix in $\gl_1$
acting on $U_\alpha$ as multiplication by $\alpha$. Obviously $U_\alpha$ splits as a direct sum of two isomorphic 1-dimensional $\gl_1$-modules.

Since our variety is 1-dimensional and $\gl_1$ acts on $U_\alpha$ by scalar matrices, any $2\times 2$ matrix
$B$ with entries in $\A$ will define a gauge field on $N(\alpha) = \A \otimes U_\alpha$. For our example we
set 
$$B = \begin{pmatrix} 0 & t \\ 1 & 0 \\ \end{pmatrix}.$$ 
Setting a basis $v_k = t^k \otimes v$, $u_k = t^k \otimes u$, the action of $W_1$ on $N(\alpha)$
can be written as
\begin{align*}
&e_n v_k = (k + \alpha n) v_{n+k} + u_{n+k}, \\
&e_n u_k = (k + \alpha n) u_{n+k} + v_{n+k+1},
\end{align*}
where $e_n = t^{n+1} \frac{\partial}{\partial t}$, $n, k \in \Z$.

The span of $\{e_1, e_0, e_{-1}\}$ forms a subalgebra in $W_1$, which is isomorphic to 
$\sll_2$. We are going to show that $N(0)$ is a simple $W_1$-module by studying its structure as a module over $\sll_2$.

It is easy to check that the Casimir element $C = e_0^2 + e_0 - e_{-1} e_1$ acts on $N(\alpha)$ as
multiplication by $\gamma = \alpha (\alpha -1)$. We can view $N(\alpha)$ as a module over the quotient algebra $U(\gamma) = U(\sll_2)/\left< C - \gamma \right>$.

R.~Block \cite{Bl} classified irreducible $\sll_2$-modules by describing maximal left ideals in the algebras $U(\gamma)$.  Simple modules are then presented as quotients by these left ideals. For general non-weight simple modules explicit realizations in terms of action on a basis are not known. 
A slightly different approach to Block's classification is given by Bavula in \cite{Ba}. We will rely on the results of
\cite{Ba} in order to understand the structure of $N(\alpha)$ as a module over $\sll_2$. 

\begin{theorem}
(a) (i) The vector $v_0 \in N(0)$ is annihilated by $s = e_{-1} e_0 - 1 \in U(0)$.

\noindent
(ii) An $\sll_2$-submodule $M \subset N(0)$ generated by $v_0$ is a simple $\sll_2$-module and
$$ M \cong U(0) / U(0) s .$$

\noindent
(iii)  The set $\left\{ v_0, \, e_0^n v_0, \, e_{-1}^n v_0 \, | \, n \geq 1 \right\}$ is a basis of $M$.

\noindent
(iv) The quotient $N(0)/M$ is a simple highest weight $\sll_2$-module with the highest weight $-1$.

\noindent
(v) $N(0)$ is irreducible as a $W_1$-module.

\

(b) Let $\alpha \not\in \frac{1}{2} \Z$, $\gamma = \alpha(\alpha -1)$.

\noindent
(i) The vector $v_0 \in N(\alpha)$ is annihilated by $p = e_1 - e_0^2 (e_0 + 1 - \alpha)$
and $q = e_{-1} e_0^2 - (e_0 + 1 - \alpha)$.

\noindent
(ii) $N(\alpha)$ is a simple $\sll_2$-module and 
$$ N(\alpha) \cong U(\gamma) / \left( U(\gamma) p + U(\gamma)q \right).$$ 

\noindent
(iii) $N(\alpha)$ is irreducible as $W_1$-module.
\end{theorem}

\begin{proof}
Let us prove part (a) of the theorem. It is straightforward to check that $(e_{-1} e_0 - 1)v_0 = 0$ in 
$N(0)$. By Corollary 3.9(b) in \cite{Ba}, the module $U(0) / U(0)s$ is a simple $\sll_2$-module. Since the Casimir element $C$ acts trivially on $N(0)$, we have a homomorphism of $\sll_2$-modules 
$\varphi: U(0) \rightarrow N(0)$, given by $\varphi(x) = x v_0$. Since $s$ annihilates $v_0$, the left ideal $U(0)s$ is in the kernel of $\varphi$ and we get a homomorphism
$$\overline\varphi : \,\,  U(0) / U(0) s \rightarrow N(0).$$
Since $U(0) / U(0) s$ is simple and $\overline\varphi (1) = v_0$, we conclude that $\overline\varphi$ is injective and its image is the $\sll_2$-submodule $M \subset N(0)$, generated by $v_0$.

Let us introduce a linear order on the basis elements of $N(0)$:
$$ \ldots < u_{-1} < v_0 < u_0 < v_1 < u_1 < v_2 < \ldots $$
This order defines the highest term and the lowest term for any non-zero element in $N(0)$. We also set
$N_+ = \Span \{ v_k, u_k \, | \, k \geq 0 \}$.

Consider the sequence $v_0, \, e_0 v_0, \,  e_0^2 v_0, \, e_0^3 v_0, \ldots$. All of these vectors are in $N_+$ and their leading terms are $v_0, u_0, v_1, u_1, \ldots$. Hence these vectors span $N_+$ and $N_+ \subset M\subset N(0)$. Since $N_+$ is invariant under the action of the Borel subalgebra spanned by $\{ e_0, e_1 \}$, we conclude that $M = \K[e_{-1}] N_+$. Moreover $e_{-1} v_k, e_{-1} u_m \in N_+$ for $k \geq 1$ and 
$m \geq 0$. Hence
$$ M = N_+ + \K[e_{-1}] v_0 .$$
The lowest terms of the vectors $e_{-1} v_0, \, e_{-1}^2 v_0, \, e_{-1}^3 v_0, \ldots$ are non-zero multiples
of $u_{-1}, u_{-2}, u_{-3}, \ldots$. Hence these vectors are linearly independent and we conclude that
$\left\{ v_0, e_0^n v_0, e_{-1}^n v_0 \, | \, n \geq 1 \right\}$ is a basis of $M$.
The images of the vectors $v_{-1}, v_{-2}, v_{-3}, \ldots$ form a basis of the quotient space $N(0) / M$.
It is easy to see that 
\begin{align*}
&e_0 v_{-1} = - v_{-1} \ \mod M, \\
&e_1 v_{-1} = 0  {\hskip 0.9cm} \mod M,
\end{align*}
and that the image of $v_{-1}$ generates $N(0) / M$ as an $\sll_2$-module. Since the Verma module for $\sll_2$ with the highest weight $-1$ is simple, we conclude that $N(0) / M$ is isomorphic to it.

Finally, let us show that $N(0)$ is simple as a $W_1$-module. We note that $e_1$ acts on $N(0)$ injectively and every non-zero element of $N(0)$ is moved into $N_+ \subset M$ by a high enough power of $e_1$. Hence every non-zero $\sll_2$ submodule in $N(0)$ contains $M$. Since $N(0)/M$ is simple, we conclude that $M$ is the only proper $\sll_2$-submodule in $N(0)$. It is easy to see that $M$ is not closed under the action of $W_1$, hence $N(0)$ is a simple $W_1$-module.
\end{proof}

 Part (b) of the theorem may be proved using Corollary 3.5 in \cite{Ba}. We omit this proof. 

% In conclusion, we remark that every simple weight $\sll_2$-module may be realized as an $\sll_2$-submodule in a tensor module for $W_1$. 

\end{document}